\documentclass[reqno]{amsart}
\usepackage{amsmath, amsthm, amssymb, mathrsfs} 
\usepackage{color, verbatim, graphicx, array} 
\usepackage{dsfont}
\usepackage{esint}                
\usepackage{xcolor}               
\usepackage[colorlinks=true]{hyperref}
\hypersetup{urlcolor=blue, citecolor=red, linkcolor=blue}
\usepackage[square,numbers]{natbib}

\allowdisplaybreaks    

\numberwithin{equation}{section}

\newtheorem{theorem}{Theorem}[section]
\newtheorem{corollary}[theorem]{Corollary}
\newtheorem{lemma}[theorem]{Lemma}
\newtheorem{prop}[theorem]{Proposition}

\theoremstyle{definition}
\newtheorem{remark}[theorem]{Remark}

\theoremstyle{definition}

\theoremstyle{definition}

 \renewcommand{\epsilon}{\varepsilon}

\begin{document}
\title[Insulated Conductivity Problem]{A simple proof for the insulated conductivity problem and application to flat boundaries}

\author[L.J. Ma]{Linjie Ma}
\address[L.J. Ma]{School of mathematics and statistics, Zhengzhou University, Zhengzhou,  450001, PR China.}
\email{ljma@amss.ac.cn}
\thanks{L.J. Ma was partially supported  by the Postdoctoral Fellowship 
Program of CPSF under Grant Number GZC20232909 and  2024M753429, NSF in China No. 11901036}



\begin{abstract}
	In high-contrast composites, the electric (or stress) field may exhibit  significant amplification in the narrow region between inclusions. The behavior of the solution depends on the distance $\epsilon$ between the inclusions, which tends to $0$. The purpose of this paper is to provide a simple proof of optimal pointwise estimates for the insulated conductivity problem in any dimension, including the case of flat inclusions. Our approach is based on two fundamental tools: the maximum principle and the Hopf lemma. A key feature of this method is that it avoids the flattening techniques commonly used in the literature, such as those in \citet{dong2021optimal,dong2022gradient}, which require transforming the narrow region into an n-dimensional cuboid. We show that the solution of the insulated problem is $\alpha$-order ($\alpha\in[0,1)$) polynomial growth for $n\geq2$ near the origin. Moreover, when the boundaries near the origin are flat, we prove that the gradient of the solution remains uniformly bounded. 
\end{abstract}
\maketitle

\section{Introduction}
\subsection{Background}
Let $D$ be a bounded open set in $\mathbb{R}^{n}$ with $n\geq2$, containing two simply connected subdomains $D_{1}$ and $D_{2}$, separated by a distance of at least $\epsilon$, where $\epsilon$ is a small positive constant. 
We consider the second-order elliptic equation in divergence form with discontinuous coefficients:
\begin{equation}\label{eq_ak-}
	\begin{cases}
		-\nabla(a(x)\nabla u)=0&\mbox{in}~D,\\
		u=g&\mbox{on}~\partial D
	\end{cases}
\end{equation}
with
\begin{equation*}
	a(x)=\mathds{1}_{D_0}+k\mathds{1}_{D_1\cup D_2},\quad D_0:=D\setminus\overline{D_{1}\cup{D}_{2}}
\end{equation*} 
where $\mathds{1}_E$ denotes the characteristic function of a set $E$.
In the context of electric conduction, the elliptic coefficient $a(x)$ is known to be directly linked to conductivity, with the solution $u$ representing the voltage potential. 
From an engineering perspective, the most significant quantity is $\nabla{u}$ which represents the electric field in the conductivity problem.
The aforementioned model has been derived from the study of composite materials by \citet{BABUSKA1999damage}, wherein they conducted numerical analysis that revealed the occurrence of high concentrations of extreme electric fields within narrow regions, specifically between adjacent inclusions.

In the case of two touching circular inclusions in two dimensions with $a(x)$ away from $0$ and $\infty$, \citet{bonnetier2000elliptic} proved that the solution belongs to $W^{1,\infty}$ for fixed $a(x)\in (0,\infty)$ in two dimensions. Later,  \citet{li2000gradient} investigated the general divergence form of second-order elliptic equations with piecewise $C^{1,\alpha}$ coefficients and demonstrated that the solution to \eqref{eq_ak-} is piecewise $C^{1,\alpha'}$ with $\alpha'\in\left(0,\frac{\alpha}{n(1+\alpha)}\right]$. This estimate was subsequently refined to $C^{1,\alpha'}$ with $\alpha'\in\left(0,\frac{\alpha}{2(1+\alpha)}\right]$ in \cite{li2003estimates}, wherein the authors considered the general second-order elliptic systems of divergence form with vector-valued functions, including systems of elasticity. For further works on this topic, see e.g. \cite{dong2015elliptic,dong2019optimal,xiong2013,Zhuge2021} and references therein.

As $k$ approaches $0$, $u$ converges to the solution of the following insulated conductivity problem, which can be expressed as follows:
\begin{equation}\label{ins}
	\begin{cases}
		\Delta u=0\quad&\mbox{in}~D_0,\\
		\frac{\partial u}{\partial\nu}=0\quad&\mbox{on}~\partial D_1\cup\partial D_2,\\
		u=g\quad&\mbox{on}~\partial D,
	\end{cases}
\end{equation}
where $\nu$ represents the inward unit normal vector.
In the context of the insulated conductivity problem, it was  demonstrated in  \cite{bao2010gradient} that the blow-up rate is equal to  be $\epsilon^{-1/2}$  in any dimension $n\geq 2$. In fact, it was proved to be optimal in \cite{ammari2005gradient,ammari2007optimal} for $n=2$.  \citet{yun2016optimal} considered two balls and derived the blow-up rate $\epsilon^{\frac{\sqrt{2}-2}{2}}$ on the $\epsilon$-segment connecting $D_1$ and $D_2$. \citet{li2023gradient} improved the upper bound in dimension $n\geq 3$ to be of order $\epsilon^{-1/2+\beta}$ for some $\beta>0$. Later, by using the Bernstein method, \citet{weinkove2023insulated} provided the blow-up rate $\gamma^*$  as the positive solution of the quadratic equation 
$	(n-2)(\gamma^*)^2+(n^2-4n+5)\gamma^*-(n^2-5n+5)=0
$
for $n\geq 4$, which improved the result presented in \cite{li2023gradient}. In a subsequent study, \citet{dong2021optimal}  considered the optimal gradient estimate and provided the optimal blow-up rate $\epsilon^{\frac{\alpha-1}{2}}$ with $\alpha=[-(n-1)+\sqrt{(n-1)^2+4(n-2)}]/2\in (0,1)$ for $n\geq3$. 
Therefore, the optimal blow-up rate of \eqref{ins} is fully determined.
In a recent study, \citet{dong2022gradient} identified the optimal blow-up rate for general elliptic equations with divergence form and demonstrated that the estimate is characterised by the first non-zero eigenvalue of a divergence form elliptic operator on $\mathbb{S}^{n-2}$. 
For further insights into the insulated conductivity problem, the reader is directed to \cite{dong2023insulated,dong2024gradient,kang2024quantitative} and their references therein.

\begin{figure}
\centering
\includegraphics[width=0.9\linewidth]{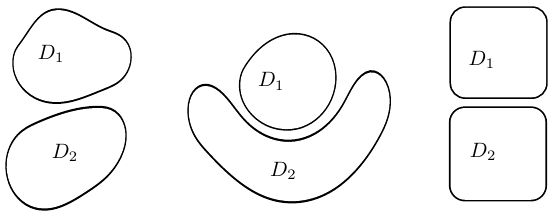} 
    \caption{Arbitrary shape inclusions}
    \label{fig}
\end{figure}

The goal of the present paper is to establish pointwise optimal estimates for the insulated conductivity problem \eqref{ins} with general inclusions in any dimensions. 
Firstly, unlike the flattening techniques employed in  \citet{dong2021optimal,dong2022gradient},  our method does not require transforming the narrow region into an $n$-dimensional cuboid. Instead, we rely directly on the maximum principle and the Hopf lemma, following ideas inspired by 
\citet{weinkove2023insulated}.  We show that the solution exhibits polynomial growth of order $\alpha$ near the origin for all $n\geq2$. 
Secondly, we consider the case of flat boundaries and prove that, in contrast to the curved case, the gradient of the solution remains uniformly bounded.

\subsection{The Domain}
 We employ the notation $x=(x',x_n)$ to represent a point in $\mathbb{R}^n$, with $x'\in\mathbb{R}^{n-1}$. Let $D$ be a bounded open  set in $\mathbb{R}^n$ that contains two subdomains $D_1$ and $D_2$, which are proximate at the origin. 

It is further assumed that the part of $\partial D_1$, $\partial D_2$ near the origin (denoted by  $\Gamma_R^\pm$, $0<R<1$) are respectively the graphs of two $C^{2,\gamma} (\gamma\in(0,1))$ functions, that is, 
\begin{equation*}
	\Gamma_R^+:=\left\lbrace x\in \mathbb{R}^n\ |\ |x'|<R, \ \ x_n=\frac{\varepsilon}{2}+h_1(x')\right\rbrace, 
\end{equation*}
\begin{equation*}
	\Gamma_R^-:=\left\lbrace x\in \mathbb{R}^n\ | \ |x'|<R, \ \ x_n=-\frac{\varepsilon}{2}+h_2(x')\right\rbrace,
\end{equation*}
where  $h_1(x')$ and $h_2(x')$ are functions in terms of $x'$ which satisfy 
\begin{equation}\label{assume1}
	h_1(x')>h_2(x')\ \ \mbox{for}\ \ 0<|x'|<R,
\end{equation}
\begin{equation}\label{assume2}
	h_1(0')=h_2(0')=0,\ \ \nabla_{x'}h_1(0')=\nabla_{x'}h_2(0')=0,
\end{equation}
\begin{equation}\label{assume3}
	\nabla_{x'}^2\left( h_1-h_2\right)(0') \geq  I,
\end{equation}
where  $a$ is some positive constant, $I$ denotes the $(n-1) \times (n-1)$ identity matrix. 

We denote
\begin{equation*}
	\Omega_R(x_0):=\left\lbrace (x',x_n)\in \mathbb{R}^n |\ \ |x'-x_0'|<R,\  -\frac{\epsilon}{2}+h_2(x')<x_n<\frac{\epsilon}{2}+h_1(x')\right\rbrace
\end{equation*}
with $x_0\in\Omega_{R}(0)$. And in the following, we denote $\Omega_R:=\Omega_R(0)$ for simplicity.
By the classical elliptic estimates, we know
\begin{equation}\label{D0}
\|u\|_{C^{1,\gamma}(D_0\setminus\Omega_{R/2})}\leq C\|g\|_{C^{1,\gamma}(\partial D)},
\end{equation}
where $C$ is a positive constant depending only on $n,\gamma,\|g\|_{C^{1,\gamma}(\partial D)}$. Hence, in the following, we focus on the following problem near the origin:
\begin{equation}\label{ins1}
\begin{cases}
-\Delta u=0\quad&\mbox{in}~\Omega_{R},\\
\frac{\partial u}{\partial\nu}=0\quad&\mbox{on}~\Gamma_R^{\pm}.
\end{cases}
\end{equation}

\subsection{Main results.}
Without loss of generality, we assume that $u(0)=0$. And in the following, we assume that the solution is nontrivial.
Firstly, we assume that
\begin{equation}\label{assume5}
h_1(x')-h_2(x')=\kappa|x'|^2+O(|x'|^{2+\gamma}),
\end{equation}
where $\kappa>0$.
Throughout the paper, we use the notation $O(A)$ to denote a quantity $Q
$ satisfying $|Q|\leq CA$, where $C$ is some universal positive constant independent of $\epsilon$.

Firstly, we have the following results.
\begin{lemma}\label{grad_local estimates}
Suppose $u\in H^1(\Omega_R)$  is a solution of system  \eqref{ins1} with $h_1(x')$ and $h_2(x')$ satisfying \eqref{assume1}-\eqref{assume3} and \eqref{assume5}, 
then there exists a positive constant $C$ which depends only on $n,\gamma,\kappa,R,\|g\|_{L^\infty(\partial D)}$, such that
\begin{equation}\label{nabla_u}
|\nabla u(x)|\leq \frac{C|u(x)|}{\sqrt{\epsilon+|x'|^2}}\ \  \mbox{for any}\ \  x\in\Omega_{R/2}. 
\end{equation}
\end{lemma}
\begin{corollary}\label{Cor:u_n}
Suppose $u\in H^1(\Omega_R)$  is a solution of system  \eqref{ins1} with $h_1(x')$ and $h_2(x')$ satisfying \eqref{assume1}-\eqref{assume3} and \eqref{assume5},
then 
    $$
    |\partial_nu(x)|\leq C\quad\mbox{for any}\ x\in\Omega_R,
    $$
    where  $C$ is some positive constant which depends only on $n,\gamma,\kappa,R,\|g\|_{L^\infty(\partial D)}$. 
\end{corollary}

Now, we give the main results.
\begin{theorem}\label{propu}
Let $u\in H^1(\Omega_R)$ be a solution of \eqref{ins1}. Then  there exists a positive constant $C$ which depends only on $n,\gamma,\kappa,R,\|g\|_{L^\infty(\partial D)}$, such that for any $x\in\Omega_{R/2}$, 
\begin{equation}\label{u:3}
\frac{1}{C}(|x'|^2+x_n^2)^{\alpha/2}\leq|u(x)|\leq C(\epsilon+|x'|^2+x_n^2)^{\alpha/2}
\quad \mbox{for}\ n\geq2  
\end{equation}
with 
$$\alpha=\alpha(n):=[-(n-1)+\sqrt{(n-1)^2+4(n-2)}]/2.
$$ 
\end{theorem}
Then the following gradient estimates  follows.
\begin{theorem}\label{grad_pointwise estimates}
Suppose $u\in H^1(\Omega_R)$  is a solution of system  \eqref{ins1} with $h_1(x')$ and $h_2(x')$ satisfying \eqref{assume1}-\eqref{assume3} and \eqref{assume5},
then there exists a positive constant $C$ which depends only on $n,\gamma,\kappa,R,\|g\|_{L^\infty(\partial D)}$, such that for $n\geq 2$,
\begin{align}\label{sup}
\frac{1}{C}(\epsilon+|x'|^2)^{\frac{\alpha-1}{2}}\leq|\nabla u(x)|\leq C(\epsilon+|x'|^2)^{\frac{\alpha-1}{2}}\ \ \mbox{in}\ \  \Omega_{R/2}.
\end{align}
\end{theorem}

If we assume that
\begin{equation}\label{assume6}
h_1(x')-h_2(x')=\mu_1x_1^2+\ldots+\mu_{n-1}x_{n-1}^2+O(|x'|^{2+\gamma}),
\end{equation}
where $\mu_i$ are some different constants. Similar to the above, we can get the following estimates.
\begin{theorem}\label{grad_pointwise estimates2}
Suppose $u\in H^1(\Omega_R)$  is a solution of system  \eqref{ins1} with $h_1(x')$ and $h_2(x')$ satisfying \eqref{assume1}-\eqref{assume3} and \eqref{assume6},
then there exists a positive constant $C$ which depends only on $n,\gamma,\kappa,R,\|g\|_{L^\infty(\partial D)}$, such that for $x\in\Omega_{R/2}$,
\begin{align*}
\frac{1}{C}(|x'|^2+x_n^2)^{\tilde\alpha/2}\leq&|u(x)|\leq C(\epsilon+|x'|^2+x_n^2)^{\tilde\alpha/2}
\quad\mbox{for}\ n\geq2,   
\end{align*}
where $\tilde{\alpha}=\tilde{\alpha}(n,\lambda):=\frac{-(n-1)+\sqrt{(n-1)^2+4\lambda}}{2}\in [0,1)$
with  $\lambda$ is the first nonzero eigenvalue of the problem 
$$
-\operatorname{div}_{\mathbb{S}^{n-2}}(a(\xi)\nabla_{\mathbb{S}^{n-2}}u)=\lambda a(\xi)u,\quad\xi\in\mathbb{S}^{n-2},
$$
here $a(\xi)>0,\ \ln a\in L^\infty(\mathbb{S}^{n-2})$. 
We have the following optimal gradient estimates:
\begin{equation*}
\frac{1}{C}(\epsilon+|x'|^2)^{\frac{\tilde{\alpha}-1}{2}}\leq|\nabla u(x)|\leq C(\epsilon+|x'|^2)^{\frac{\tilde{\alpha}-1}{2}}\ \ \mbox{in}\ \  \Omega_{R/2}.   
\end{equation*}
\end{theorem}
\begin{remark}
     In \cite{dong2021optimal,dong2022gradient}, the authors provide a lower bound for the $C^2$ domain $D_0$ and $C^4$ relatively strictly convex open sets $D_1$, $D_2$ that are axially symmetric with respect to the $x_n$-axis. They also assume that the function $g$ is a special odd function $x_j$. In  Theorem \ref{grad_pointwise estimates} and \ref{grad_pointwise estimates2}, we only assume that $D_1$ and $D_2$ are $C^{2,\gamma}$ inclusions. It is evident that our condition is less stringent than that of \cite{dong2021optimal,dong2022gradient} and our result holds for $g$ that belongs to $C^{1,\gamma}(\partial D)$.     
\end{remark}
\begin{remark}
In fact, we only need to assume that $g\in C^0(\partial D)$. Indeedly, taking a slightly smaller domain $\tilde{D}_0\subset D_0$, we see that $\tilde{g}:=u|_{\partial \tilde{D}_0}$ satisfies
$\|\tilde g\|_{C^{1,\gamma}(\partial\tilde{D})}\leq C\|u\|_{L^\infty (\tilde D_0)}\leq C\|g\|_{C^0(\partial D)}$ in view of the interior estimates and the maximum principle for harmonic functions. The desired results follows by working on the domain $\tilde{D}_0$ with boundary data $\tilde{g}$.
\end{remark}

At last, we give the estimates for the flat boundaries, that is
\begin{equation}\label{assume4}
h_1(x')=h_2(x')=0.
\end{equation}
\begin{theorem}\label{flat}
Suppose $u\in H^1(\Omega_R)$  is a solution of system  \eqref{ins1} with $h_1(x')$ and $h_2(x')$ satisfying \eqref{assume4},
then there exists a positive constant $C$ which depends only on $n,\gamma,\kappa,\|g\|_{L^\infty(\partial D)}$, such that for $n\geq 2$,
\begin{align}\label{grad_flat}
\frac{1}{C}\sqrt{|x'|^2+x_n^2}\leq |u(x)|\leq C\sqrt{\epsilon+|x'|^2+x_n^2}\quad\mbox{in}\ \  \Omega_{R/2}.
\end{align}
and
\begin{equation*}
|\nabla u(x)|\leq C\quad\mbox{in}\ \  \Omega_{R/2}.   
\end{equation*}
\end{theorem}

The outline of the paper is as follows. In section \ref{sec2}, we give the proof of Theorem \ref{propu}. The proof of Theorem \ref{grad_pointwise estimates} and Theorem \ref{grad_pointwise estimates2} will be given in section \ref{sec3}. In section \ref{sec4}, by using the method in \cite{weinkove2023insulated}, we  give the proof of Lemma \ref{grad_local estimates} for general inclusions. At last section, we will  prove Theorem \ref{flat} in section \ref{sec5}.

\section{Proof of Theorem \ref{propu} }\label{sec2}

For $n=2$, we deal with the equation \eqref{ins1} directly.
\begin{lemma}\label{u_n=2}
Let $u\in H^1(\Omega_R)$ be a solution of \eqref{ins1}. Then there exists a positive constant $C$ which depends only on $n,\gamma,\kappa,R,\|g\|_{L^\infty(\partial D)}$, such that for $n=2$,
\begin{equation}\label{u-n=2}
\frac{1}{C}\ln\left(\frac{\epsilon+x_1^2+x_2^2}{\epsilon}\right)\leq|u(x)|\leq \frac{C\sqrt{\epsilon}|x_1|}{\epsilon+x_1^2+x_2^2}\quad\mbox{in}\ \Omega_{\sqrt{\epsilon}}.
\end{equation}
\end{lemma}
We will give the proof of Lemma \ref{u_n=2} later.

For $n\geq3$, we denote $Y_{k,i}$ to be a $k$-th degree normalized spherical harmonics so that $\{Y_{k,i}\}_{k,i}$ forms an orthonormal basis of $L^2(\mathbb{S}^{n-2})$. Let
\begin{equation}\label{def_uhat}
\hat{u}(r,x_n):=\hat{u}_{k,i}(r,x_n)=\int_{\mathbb{S}^{n-2}}u(r,\xi,x_n)Y_{k,i}(\xi)d \xi,
\end{equation}
then $\hat{u}$ is the solution of the following
\begin{equation}\label{eq:uhat}
\begin{cases}
\hat{u}_{rr}+\frac{n-2}{r}\hat{u}_r-\frac{k(k+n-3)}{r^2}\hat{u}+\hat{u}_{nn}=0 &\mbox{in}\ \tilde{\Omega}_{R},\\
\frac{\partial\hat{u}}{\partial\nu}=0&\mbox{on}\ \tilde{\Gamma}_R^\pm,\\
\hat{u}(0)=0,
\end{cases}
\end{equation}
where
\begin{equation*}
\tilde{\Omega}_R:=\left\lbrace (r,x_n)\in \mathbb{R}^2\ |\ 0<r<R,\  -\frac{\epsilon}{2}+h_2(r)<x_n<\frac{\epsilon}{2}+h_1(r)\right\rbrace,
\end{equation*}
\begin{equation*}
\tilde{\Gamma}_R^+:=\left\lbrace (r,x_n)\in \mathbb{R}^2\ |\ 0<r<R,\ x_n=\frac{\varepsilon}{2}+h_1(r)\right\rbrace, 
\end{equation*}
\begin{equation*}
\tilde{\Gamma}_R^-:=\left\lbrace (r,x_n)\in \mathbb{R}^2\ |\ 0<r<R,\  x_n=-\frac{\varepsilon}{2}+h_2(r) \right\rbrace.
\end{equation*}
By standard properties of spherical harmonic expansions in $H^1$, each coefficient $\hat u_{k,i}$ belongs to $H^1(\tilde{\Omega}_R)$.
We have the following estimates for $\hat{u}$:
\begin{prop}\label{propuhat}
Let $\hat{u}\in H^1(\tilde{\Omega}_R)$ be a solution of \eqref{eq:uhat}. Then there exists a positive constant $C$ which depends only on $n,\gamma,\kappa,R,\|g\|_{L^\infty(\partial D)}$, such that  
for $n\geq3$,
\begin{equation}\label{uhat}
\frac{1}{C}(r^2+x_n^2)^{\alpha_k/2}\leq|\hat{u}(r,x_n)|\leq C\left[(r^2+x_n^2)^{\alpha_k/2}+\epsilon\right]
\quad \mbox{in}\ \tilde{\Omega}_R,
\end{equation}
where 
\begin{equation}\label{alphak}
\alpha_k:=\frac{-(n-1)+\sqrt{(n-1)^2+4k(k+n-3)}}{2}.   
\end{equation}
\end{prop}
\begin{proof}
By considering $\pm\hat{u}$,  without loss of generality we may assume that $\hat{u}\geq0$ in $\tilde{\Omega}_R$. And by taking suitable $R$, we can assume that $\inf_{x_n\in(-\frac{\epsilon}{2}+h_2(R),\frac{\epsilon}{2}+h_1(R)}\hat{u}\neq0$.

Since the operator is uniformly elliptic in $\tilde{\Omega}\setminus\{r=0\}$ and the Neumann boundary is $C^{2,\gamma}$, by Hopf lemma (see e.g. Lemma 3.4 in \cite{gilbarg1977elliptic}), we know $\hat{u}$ cannot attain the maximum and minimum on $\tilde{\Gamma}_R^\pm$, hence, by the weak maximum principle (see e.g. Corollary 3.2 in \cite{gilbarg1977elliptic}), one has
\begin{equation}\label{max:uhat}
0\leq\inf_{\partial\tilde{\Omega}_{R}\backslash\tilde{\Gamma}_R^{\pm}}\hat{u}\leq\inf_{\partial\tilde{\Omega}_R}\hat{u}\leq\hat{u}\leq\sup_{\partial\tilde{\Omega}_R}\hat{u}\leq\sup_{\partial\tilde{\Omega}_{R}\backslash\tilde{\Gamma}_R^{\pm}}\hat{u}\quad\mbox{in}\ \tilde{\Omega}_R.
\end{equation}
Next, define the operator $L$ as
$$
Lu:=u_{rr}+\frac{n-2}{r}u_r-\frac{k(k+n-3)}{r^2}u+u_{nn}.
$$
We intend to find a supersolution $\phi(r,x_n)$ which satisfies 
\begin{equation}\label{eq:phi}
L\phi\leq0\quad\mbox{in}\ \tilde{\Omega}_R,\quad
\frac{\partial\phi}{\partial\nu}\geq0\quad
\mbox{on}\ \tilde{\Gamma}_R^\pm  . 
\end{equation}
and then using the Hopf lemma and the maximum principle to give the upper bound \eqref{uhat}.

In the following, we give a function $\phi$ which satisfies \eqref{eq:phi}.
Firstly, we assume that
\begin{equation}\label{h_1h_2}
h_1(x')=\kappa_1|x'|^2+O(|x'|^{2+\gamma}),\quad
h_2(x')=\kappa_2|x'|^2+O(|x'|^{2+\gamma}),
\end{equation}
where $\kappa_1-\kappa_2=\kappa>0$.  
Then we divide the proof into two cases:

\textbf{Case 1: $\kappa_1>0$, $\kappa_2<0$.}
For any $(r,x_n)\in\tilde{\Omega}_R$, let
\begin{equation*}
\phi=\phi(r,x_n):=
(r^2+2x_n^2)^{\alpha_k/2}+r^\beta(r^2+bx_n^2)^{\xi/2}\quad\mbox{for}\ n\geq3,     
\end{equation*}
where the parameter $\xi, \beta$ and $b$  satisfy
\begin{equation}\label{alpha_0}
0<\xi<\frac{-(n-3+b+2\beta)+\sqrt{(n-3+b+2\beta)^2+4[k(k+n-3)-(n-3+\beta)\beta]}}{2},    
\end{equation}
\begin{equation}\label{parameter_alpha0}
\alpha_k-\xi<\beta<\min\{\alpha_k-\xi+\gamma, 1\},   
\end{equation}
\begin{equation}\label{parameter_b}
 b>2+\frac{2\beta}{\xi} .   
\end{equation}
Obviously, $\xi<\alpha_k$, $\phi\in C^2\left((0,R]\times\left[-\frac{\epsilon}{2}+h_2(r),\frac{\epsilon}{2}+h_1(r)\right]\right)$ and $\phi(0)=0$. 
In the following, we show that $\phi$ satisfies \eqref{eq:phi}.

Denote 
$$
\phi_1:=(r^2+2x_n^2)^{\alpha_k/2},\quad
\phi_2:=r^\beta(r^2+bx_n^2)^{\xi/2}.
$$
Then using \eqref{alphak}, we have
\begin{align}\label{Lphi_1}
L\phi_1=&[{\alpha_k}^2+(n-1){\alpha_k}-k(k+n-3)](r^2+2x_n^2)^{\frac{{\alpha_k}}{2}-1}\nonumber\\
&+2\alpha_k({\alpha_k}-2)(r^2+2x_n^2)^{\frac{{\alpha_k}}{2}-2}x_n^2\nonumber\\
&-\frac{2k(k+n-3)}{r^2}(r^2+2x_n^2)^{\frac{{\alpha_k}}{2}-1}x_n^2\nonumber\\
=&2\alpha_k({\alpha_k}-2)(r^2+2x_n^2)^{\frac{{\alpha_k}}{2}-2}x_n^2
-\frac{2k(k+n-3)}{r^2}(r^2+2x_n^2)^{\frac{{\alpha_k}}{2}-1}x_n^2
\   \mbox{in}\ \tilde{\Omega}_R. 
\end{align}
Similarly,
\begin{align}\label{Lphi_2}
L\phi_2=&[\xi^2+(n-3+b+2\beta)\xi-k(k+n-3)]r^\beta(r^2+bx_n^2)^{\frac{\xi}{2}-1}\nonumber\\
&+(n-3+\beta)\beta r^{\beta-2}(r^2+bx_n^2)^{\frac{\xi}{2}}\nonumber\\
&+b(b-1)\alpha_k(\alpha_k-2)r^{\beta}(r^2+bx_n^2)^{\frac{\xi}{2}-2}x_n^2
-b(n-2)r^{\beta-2}(r^2+bx_n^2)^{\frac{\xi}{2}-1}x_n^2\nonumber\\
=&\Big[\xi^2+(n-3+b+2\beta)\xi-k(k+n-3)+(n-3+\beta)\beta\Big]
r^\beta(r^2+bx_n^2)^{\frac{\xi}{2}-1}\nonumber\\
&+b(b-1)\xi(\xi-2)r^{\beta}(r^2+bx_n^2)^{\frac{\xi}{2}-2}x_n^2\nonumber\\
&+[(n-3+\beta)\beta -k(k+n-3)]br^{\beta-4}(r^2+bx_n^2)^{\frac{\xi}{2}-1}x_n^2
\quad\mbox{in}\ \tilde{\Omega}_R.  
\end{align}
Since $\xi$ satisfies \eqref{alpha_0}, we know the first term in \eqref{Lphi_2} is non-positive. And the last two terms in \eqref{Lphi_1} and \eqref{Lphi_2} are much smaller compared to the first term in \eqref{Lphi_2} for any $r\in(0,R)$, that means these lower-order terms are $o(r^{\beta+\xi-2})$ and can be absorbed by the leading positive term for $r\in(0,R)$, provided $R$ is chosen sufficiently small. Hence, we  have
$$
L\phi=L\phi_1+L\phi_2\leq0\quad \mbox{in}\ \tilde{\Omega}_R. 
$$

Next, we consider the boundaries. In view of the assumption \eqref{h_1h_2}, the unit outer normal vector is 
$$
\nu|_{\tilde{\Gamma}_R^+}=\frac{1}{\sqrt{1+(h_1')^2}}(-2\kappa_1r+O(r^{1+\gamma}),1),
$$
$$ 
\nu|_{\tilde{\Gamma}_R^-}=\frac{1}{\sqrt{1+(h_2')^2}}(2\kappa_2r+O(r^{1+\gamma}),-1).
$$
For the boundary value, using the fact that $x_n=\frac{\epsilon}{2}+\kappa_1r^2+O(r^{2+\gamma})$ on $\tilde\Gamma_R^+$, one has
\begin{align*}
\frac{\partial\phi_1}{\partial\nu}\Big|_{\tilde{\Gamma}_R^+}
=&\frac{1}{\sqrt{1+(h_1')^2}}\Big\{\alpha_k(r^2+2x_n^2)^{\frac{{\alpha_k}}{2}-1}r \left[-2\kappa_1r+O(r^{1+\gamma})\right]\nonumber\\
&+2\alpha_k(r^2+2x_n^2)^{\frac{{\alpha_k}}{2}-1}\left[\frac{\epsilon}{2}+\kappa_1r^2+O(r^{2+\gamma})\right]\Big\}\nonumber\\
=&\frac{1}{\sqrt{1+(h_1')^2}}\Big\{\alpha_k\epsilon(r^2+2x_n^2)^{\frac{{\alpha_k}}{2}-1}+
(r^2+2x_n^2)^{\frac{{\alpha_k}}{2}-1}\cdot O(r^{2+\gamma})\Big\}.
\end{align*}
Similarly,
\begin{align*}
\frac{\partial\phi_2}{\partial\nu}\Big|_{\tilde{\Gamma}_R^+}
=&\frac{1}{\sqrt{1+(h_1')^2}}\Big\{\left[\xi r^{\beta+1}(r^2+bx_n^2)^{\frac{\xi}{2}-1}+\beta r^{\beta-1}(r^2+bx_n^2)^{\frac{\xi}{2}}\right]\nonumber\\
&\cdot\left[-2\kappa_1r+O(r^{1+\gamma})\right]
+b\xi r^\beta(r^2+bx_n^2)^{\frac{\xi}{2}-1}\left[\frac{\epsilon}{2}+\kappa_1r^2+O(r^{2+\gamma})\right]\Big\}\nonumber\\
=&\frac{1}{\sqrt{1+(h_1')^2}}\Big\{\frac{b\xi}{2}\epsilon r^\beta(r^2+bx_n^2)^{\frac{\xi}{2}-1}
+[(b-2)\xi-2\beta]\kappa_1r^{\beta+2}(r^2+bx_n^2)^{\frac{\xi}{2}-1}\nonumber\\
&+r^{\beta}(r^2+bx_n^2)^{\frac{\xi}{2}-1}\cdot
O(r^{2+\gamma}+x_n^2)\Big\}.
\end{align*}
Hence,
\begin{align*}
\frac{\partial\phi}{\partial\nu}\Big|_{\tilde{\Gamma}_R^+}
=&\frac{\partial\phi_1}{\partial\nu}\Big|_{\tilde{\Gamma}_R^+}
+\frac{\partial\phi_2}{\partial\nu}\Big|_{\tilde{\Gamma}_R^+}\nonumber\\
=&\frac{1}{\sqrt{1+(h_1')^2}}\Big\{\alpha_k\epsilon(r^2+2x_n^2)^{\frac{{\alpha_k}}{2}-1}+\frac{b\xi}{2}\epsilon r^\beta(r^2+bx_n^2)^{\frac{\xi}{2}-1}\nonumber\\
&+[(b-2)\xi-2\beta]\kappa_1r^{\beta+2}(r^2+bx_n^2)^{\frac{\xi}{2}-1}
+(r^2+2x_n^2)^{\frac{{\alpha_k}}{2}-1}\cdot O(r^{2+\gamma})\nonumber\\
&+r^\beta(r^2+bx_n^2)^{\frac{\xi}{2}-1}\cdot
O(r^{2+\gamma}+ x_n^2)
\Big\}.
\end{align*}
Since $b$ and $\xi$ satisfy \eqref{parameter_alpha0} and \eqref{parameter_b} respectively, since $R$ is fixed and small depending only on $n,\gamma,\kappa$, we know the term $(r^2+2x_n^2)^{\frac{{\alpha_k}}{2}-1}\cdot O(r^{2+\gamma})$ and $r^\beta(r^2+bx_n^2)^{\frac{\xi}{2}-1}\cdot
O(r^{2+\gamma}+ x_n^2)$ can be controlled by the other positive terms. Hence, we have 
$$
\frac{\partial\phi}{\partial\nu}\Big|_{\tilde{\Gamma}_R^+}\geq0.
$$
Similarly,  one has 
\begin{align*}
\frac{\partial\phi}{\partial\nu}\Big|_{\tilde{\Gamma}_R^-}
=&\frac{1}{\sqrt{1+(h_2')^2}}\Big\{\alpha_k\epsilon(r^2+2x_n^2)^{\frac{{\alpha_k}}{2}-1}+\frac{b\xi}{2}\epsilon(r^2+bx_n^2)^{\frac{\xi}{2}-1}\nonumber\\
&+[(b-2)\xi-2\beta]\kappa_2r^{\beta+2}(r^2+bx_n^2)^{\frac{\xi}{2}-1}\nonumber\\
&+
(r^2+2x_n^2)^{\frac{{\alpha_k}}{2}-1}\cdot O(r^{2+\gamma})+r^\beta(r^2+bx_n^2)^{\frac{\xi}{2}-1}\cdot
O(r^{2+\gamma}+ x_n^2)\Big\}
\geq 0.
\end{align*}
Hence, we get \eqref{eq:phi}.

\textbf{Case 2}: $\kappa_1>0$, $\kappa_2\geq0$.
We modify the definition of $\phi$ as follows:
\begin{align*}
\phi=&(r^2+2x_n^2)^{{\alpha_k}/2}\nonumber\\
&+r^\beta\left[r^2+b_1\left(x_n-\frac{\epsilon}{2}-h_1(r)\right)^2+b_2\left(x_n+\frac{\epsilon}{2}-h_2(r)\right)^2\right]^{\frac{\xi}{2}}
\end{align*}
with
\begin{align*}
&0<\xi\nonumber\\
&<\frac{-(n-3+b_1+b_2+2\beta)+\sqrt{(n-3+b_1+b_2+2\beta)^2+4[k(k+n-3)-(n-3+\beta)\beta]}}{2},
\end{align*}
\begin{equation*}
\alpha_k-\xi<\beta<\min\{\alpha_k-\xi+\gamma,1\},    
\end{equation*}
\begin{equation*}
\quad b_1>\frac{2(\beta+\xi)\kappa_1}{\xi\kappa},\quad b_2>\frac{2(\beta+\xi)\kappa_2}{\xi\kappa} .   
\end{equation*}
Then \eqref{eq:phi} can be obtained by similar discussion as in Case 1.

Next, in view of \eqref{eq:phi}, since $\tilde{\Gamma}_R^\pm$ is $C^{2,\gamma}$, by Hopf lemma and the maximum principle, we have
\begin{equation}\label{inf:phi}
\inf_{\tilde{\Omega}_R}\phi\geq\inf_{\partial\tilde{\Omega}_{R}\backslash\tilde{\Gamma}_R^{\pm}}\phi^- \geq0.
\end{equation}
Let
$$
\underline{u}:=\hat u-\mathcal{A}\phi
$$
with
\begin{equation*}
\mathcal{A}=\frac{\sup_{x_n\in(-\frac{\epsilon}{2}+h_2(R),\frac{\epsilon}{2}+h_1(R))}\hat{u}(R,x_n)}{\inf_{x_n\in(-\frac{\epsilon}{2}+h_2(R),\frac{\epsilon}{2}+h_1(R))}\phi(R,x_n)}>0.
\end{equation*}
In view of \eqref{max:uhat} and \eqref{inf:phi}, one has
\begin{equation}\label{sup:underline_u}
\sup_{\tilde{\Omega}_R}\underline{u}\leq\sup_{\partial\tilde{\Omega}_{R}\backslash\tilde{\Gamma}_R^{\pm}}\hat{u}-\mathcal{A}\inf_{\partial\tilde{\Omega}_{R}\backslash\tilde{\Gamma}_R^{\pm}}\phi\leq \sup_{\partial\tilde{\Omega}_{R}\backslash\tilde{\Gamma}_R^{\pm}}\underline{u}.    
\end{equation}
We can show that 
\begin{equation}\label{u+}
\sup_{\partial\tilde{\Omega}_{R}\backslash\tilde{\Gamma}_R^{\pm}}\underline{u}\leq C\epsilon.   
\end{equation}
In fact, in view of Corollary \ref{Cor:u_n},
$$
|\partial_n\hat u(r,x_n)|\leq C \quad\mbox{in}\ \tilde{\Omega}_R.
$$
Hence,
$$
\hat{u}(0,x_n)=\hat{u}(0,x_n)-\hat{u}(0)\leq C\max_{x_n\in(-\epsilon/2,\epsilon/2)}|\partial_n\hat{u}(0,x_n)||x_n|\leq C\epsilon,
$$
so that 
\begin{align}\label{u(0)}
\underline{u}(0,x_n)= \hat{u}(0,x_n)-\mathcal{A}\phi(0,x_n)
\leq\hat{u}(0,x_n)
\leq C\epsilon.
\end{align}
On the other hand,
\begin{equation}\label{u(1)}
\underline{u}(R,x_n)=\hat{u}(R,x_n)-\mathcal{A}\phi(R,x_n)\leq0 .   
\end{equation}
In view of \eqref{u(0)} and \eqref{u(1)}, we know \eqref{u+} holds.

Hence, in view of \eqref{sup:underline_u} and \eqref{u+}, one has
\begin{equation*}
\hat u(r,x_n)\leq \mathcal{A}\phi+C\epsilon\leq 
C\left(r^{{\alpha_k}}+\epsilon\right)\quad\mbox{in}\ \tilde{\Omega}_R,
\end{equation*}
where $C$ depends only on $n,R,\kappa,\|g\|_{L^\infty(\partial D)}$. The upper bound is proved.

For the lower bound, similarly, let
\begin{equation*}
\tilde\phi(r,x_n):=r^{\beta_1}(r^2+4x_n^2)^{\frac{\beta_2}{2}}\quad\mbox{in}\ \tilde\Omega_R,     
\end{equation*}
with 
$$
\beta_1+\beta_2=\alpha_k,\quad\beta_2>\beta_1>0
$$
We get that
\begin{align}\label{Ltildephi}
L\tilde\phi
=&\Big[\beta_2^2+(n+1+2\beta_1)\beta_2-k(k+n-3)+(n-3+\beta_1)\beta_1\Big]
r^{\beta_1}(r^2+4x_n^2)^{\frac{\beta_2}{2}-1}\nonumber\\
&+12\beta_2(\beta_2-2)r^{\beta_1}(r^2+4x_n^2)^{\frac{\beta_2}{2}-2}x_n^2\nonumber\\
&+4[(n-3+\beta_1)\beta_1 -k(k+n-3)]r^{\beta_1-4}(r^2+4x_n^2)^{\frac{\beta_2}{2}-1}x_n^2\nonumber\\
=&2(\beta_2-\beta_1)
r^{\beta_1}(r^2+4x_n^2)^{\frac{\beta_2}{2}-1}
+12\beta_2(\beta_2-2)r^{\beta_1}(r^2+4x_n^2)^{\frac{\beta_2}{2}-2}x_n^2\nonumber\\
&+4[(n-3+\beta_1)\beta_1 -k(k+n-3)]r^{\beta_1-4}(r^2+4x_n^2)^{\frac{\beta_2}{2}-1}x_n^2
\quad\mbox{in}\ \tilde{\Omega}_R.  
\end{align}
Since $|x_n|\approx\epsilon+|x'|^2$, the last two term in \eqref{Ltildephi} is lower terms and can be absorbed by the first negative term. Hence, we have
$$
L\tilde\phi\leq0\quad \mbox{in}\ \tilde{\Omega}_R. 
$$
Since $\tilde\phi$ is increasing respect to $r$ and $|x_n|$, we know
\begin{equation}\label{eq:-phi}
0\leq\sup_{\tilde{\Omega}_R}\tilde\phi\leq\sup_{\partial\tilde{\Omega}_{R}\backslash\tilde{\Gamma}_R^{\pm}}\tilde\phi.
\end{equation}
Next, we denote
$$
\bar{u}:=\hat{u}-\mathcal{B}\tilde\phi
$$
with
\begin{align*}
\mathcal{B}:=\frac{\inf_{x_n\in(-\frac{\epsilon}{2}+h_2(R),\frac{\epsilon}{2}+h_1(R))}\hat{u}(R,x_n)}{\sup_{x_n\in(-\frac{\epsilon}{2}+h_2(R),\frac{\epsilon}{2}+h_1(R))}\tilde\phi(R,x_n)}>0.
\end{align*}
In view of \eqref{max:uhat} and \eqref{eq:-phi}, one has
\begin{equation}\label{inf_ubar}
\inf_{\tilde{\Omega}_R}\bar u
\geq\inf_{\partial\tilde{\Omega}_R\backslash\tilde{\Gamma}_R^\pm}\hat{u}-\mathcal{B}\sup_{\partial\tilde{\Omega}_R\backslash\tilde{\Gamma}_R^\pm}\tilde\phi\geq\inf_{\partial\tilde{\Omega}_R\backslash\tilde{\Gamma}_R^\pm}\bar{u}.    
\end{equation}
Since  
\begin{equation*}
\bar{u}(0,x_n)= \hat{u}(0,x_n)-\mathcal{B}\tilde\phi(0,x_n)
\geq0,
\end{equation*}
and by the definition of $\mathcal{B}$, one has
\begin{equation*}
\bar u(R,x_n)\geq\hat{u}(R,x_n)-\mathcal{B}\tilde\phi(R,x_n)
\geq 0,
\end{equation*}
we have
$$
\inf_{\partial\tilde{\Omega}_R\backslash\tilde{\Gamma}_R^\pm}\bar{u}\geq0.
$$
Hence, from \eqref{inf_ubar}, we can deduce that
$$
\bar{u}\geq\inf_{\tilde{\Omega}_R}\bar u\geq0,
$$
that is
$$
\hat{u}(r,x_n)\geq\mathcal{B}\tilde\phi(r,x_n)
\geq\frac{1}{C}(r^2+x_n^2)^{\alpha_k/2}\quad\mbox{in}\ \tilde{\Omega}_R.
$$
The proof is completed.
\end{proof}
\begin{remark}
    From the proof of Proposition \ref{propuhat}, we know that $\phi_1=(r^2+2x_n^2)^{\alpha_k/2}$ is the leading term of $\hat{u}$ and the index $\alpha_k$ is optimal.
\end{remark}
\begin{remark}
    Note that $\alpha_k\sim k$, as $k\rightarrow\infty$, which guarantees convergence of the series. And in Proposition \ref{propuhat}, the constant $C$ is independent of $k$.
\end{remark}

\begin{proof}[Proof of Proposition \ref{u_n=2}]
The proof is similar to Proposition \ref{propuhat}.

\textbf{Case 1: $\kappa_1=-\kappa_2>0$.}
We consider the function $\varphi(x)$ in half plane $\tilde{\Omega}_d$, where $0<d<R$ will be determined later.
\begin{equation*}
\varphi(x)=\frac{\sqrt{\sigma}x_1}{x_1^2+(x_2+\sqrt{\sigma})^2}+\frac{\sqrt{\sigma}x_1 }{x_1^2+(x_2-\sqrt{\sigma})^2},\quad\sigma=\frac{\epsilon}{2\kappa_1}
\end{equation*}
which satisfies
\begin{equation*}
\begin{cases}
\Delta\varphi=0&\mbox{in}\ \tilde\Omega_d,\\
\frac{\partial\varphi}{\partial\nu}\geq 0&\mbox{on}\    \tilde\Gamma_d^\pm.
\end{cases}   
\end{equation*}
Obviously, $\varphi$ is harmonic function.
For the boundary value, we denote 
$$
D_1=x_1^2+(x_2+\sqrt{\sigma})^2,\quad
D_2=x_1^2+(x_2-\sqrt{\sigma})^2.
$$
Using the fact that $x_2=\frac{\epsilon}{2}+\kappa_1x_1^2+O(x_1^{2+\gamma})$ on $\tilde{\Gamma}_R^+$, one has
\begin{align*}
\frac{\partial\varphi_1}{\partial\nu}\Big|_{\tilde{\Gamma}_R^+}
=\frac{\sqrt{\sigma}}{\sqrt{1+(h_1')^2}}&\Bigg\{\left[\frac{(x_2+\sqrt{\sigma})^2-x_1^2}{D_1^2}+\frac{(x_2-\sqrt{\sigma})^2-x_1^2}{D_2^2}\right]\left[-2\kappa_1x_1+O(x_1^{1+\gamma})\right]\nonumber\\
&-\frac{2x_1(x_2+\sqrt{\sigma})}{D_1^2}-\frac{2x_1(x_2-\sqrt{\sigma})}{D_2^2}\Bigg\}\nonumber\\
=\frac{\sqrt{\sigma}}{\sqrt{1+(h_1')^2}}&\Bigg\{\frac{-2\kappa_1x_1(x_2+\sqrt{\sigma})^2-2x_1(\sqrt{\sigma}+\frac{\epsilon}{2})}{D_1^2}\nonumber\\
&+\frac{-2\kappa_1x_1(x_2-\sqrt{\sigma})^2-2x_1(\frac{\epsilon}{2}-\sqrt{\sigma})}{D_2^2}
+O\left(\frac{x_1^{1+\gamma}}{\epsilon+x_1^2}\right)
\Bigg\}\nonumber\\
=\frac{\sqrt{\sigma}}{\sqrt{1+(h_1')^2}}&\Bigg\{-(2\kappa_1\sigma+\epsilon) x_1\left(\frac{1}{D_1^2}+\frac{1}{D_2^2}\right)-2x_1\sqrt{\sigma}\left(\frac{1}{D_1^2}-\frac{1}{D_2^2}\right)\nonumber\\
&+O\left(\frac{x_1^{1+\gamma}}{\epsilon+x_1^2}\right)
\Bigg\}
\end{align*}
Since
\begin{align*}
\frac{1}{D_1^2}-\frac{1}{D_2^2}
=-\frac{(D_1+D_2)(D_1-D_2)}{D_1^2D_2^2}
=&-\frac{4\sqrt{\sigma} x_2(D_1+D_2)}{D_1^2D_2^2}\nonumber\\
=&-4\sqrt{\sigma} \left(\frac{x_2/D_2}{D_1^2}+\frac{x_2/D_1}{D_2^2}\right),
\end{align*}
and  by the assumption $\sigma=\frac{\epsilon}{2\kappa_1}$,
\begin{equation*}
 \frac{x_2}{D_1}=\frac{\frac{\epsilon}{2}+\kappa_1x_1^2+O(x_1^{2+\gamma})}{\sigma+x_1^2+O(\sqrt{\sigma}x_2)}=\kappa_1+O(\sqrt{\epsilon}+x_1^\gamma)\ \ \mbox{on}\ \tilde\Gamma_d^+.
\end{equation*}
Similarly,
\begin{equation*}
\frac{x_2}{D_2}=\kappa_1+O(\sqrt{\epsilon}+x_1^\gamma)\ \ \mbox{on}\ \tilde\Gamma_d^+.
\end{equation*}
From the above, one has
\begin{equation*}
\frac{1}{D_1^2}-\frac{1}{D_2^2}
=-4\sqrt{\sigma}\left(\kappa_1+O(\sqrt{\epsilon}+x_1^\gamma)\right)\left(\frac{1}{D_1^2}+\frac{1}{D_2^2}\right).
\end{equation*}
Thus,
\begin{align*}
\frac{\partial\varphi_1}{\partial\nu}\Big|_{\tilde{\Gamma}_d^+}
=\frac{\sqrt{\sigma}}{\sqrt{1+(h_1')^2}}\Bigg\{2\epsilon x_1\left(1+O(\sqrt{\epsilon}+x_1^\gamma)\right)\left(\frac{1}{D_1^2}+\frac{1}{D_2^2}\right)+O\left(\frac{x_1^{1+\gamma}}{\epsilon+x_1^2}\right)
\Bigg\}
\geq0.
\end{align*}
The value of $\frac{\partial\tilde\varphi}{\partial\nu}$ in $\tilde\Gamma_d^-$  can be proved similar.

For the lower bound, we consider the function in $\tilde\Omega_d$,
\begin{equation*}
\tilde\varphi(x)=\ln(x_1^2+(x_2+\sqrt{\sigma})^2)+\ln(x_1^2+(x_2-\sqrt{\sigma})^2)-2\ln\sigma,\ \sigma=\frac{\epsilon}{2\kappa_1}.
\end{equation*}
which satisfies
\begin{equation*}
\begin{cases}
\Delta\tilde\varphi=0&\mbox{in}\ \tilde\Omega_d,\\
\frac{\partial\tilde\varphi}{\partial\nu}\leq 0&\mbox{on}\    \tilde\Gamma_d^\pm.
\end{cases}   
\end{equation*}
In fact, obviously, $\tilde\varphi$ is a harmonic function in $\tilde\Omega_d$.
For the boundary condition, 
\begin{align*}
&\frac{\partial\tilde\varphi}{\partial\nu}\Big|_{\tilde{\Gamma}_d^+}\nonumber\\
=&\frac{1}{\sqrt{1+(h_1')^2}}\left\{\left(
\frac{1}{D_1}+\frac{1}{D_2}\right)2x_1
[-2\kappa_1x_1+O(x_1^{1+\gamma})]+\frac{2(x_2+\sqrt{\sigma})}{D_1}+\frac{2(x_2-\sqrt{\sigma})}{D_2}\right\}\nonumber\\
=&\frac{1}{\sqrt{1+(h_1')^2}}\left\{
\left(-4\kappa_1x_1^2+2x_2+O(x_1^{2+\gamma})\right)\left(\frac{1}{D_1}+\frac{1}{D_2}\right)+2\sqrt{\sigma}\left(\frac{1}{D_1}-\frac{1}{D_2}\right)
\right\}\nonumber\\
=&\frac{1}{\sqrt{1+(h_1')^2}}\left\{
\left(-2\kappa_1x_1^2+\epsilon+O(x_1^{2+\gamma})\right)\left(\frac{1}{D_1}+\frac{1}{D_2}\right)+2\sqrt{\sigma}\left(\frac{1}{D_1}-\frac{1}{D_2}\right)
\right\}.
\end{align*}  
Since
\begin{align*}
2\sqrt{\sigma}\left(\frac{1}{D_1}-\frac{1}{D_2}\right)=\frac{-8\sigma x_2}{D_1D_2}=\frac{-8\sigma x_2}{(\sigma+x_1^2)^2}+O(\epsilon)
=&\frac{-8\sigma\kappa_1+O(x_1^{2+\gamma})}{\sigma+x_1^2}+O(\epsilon)\nonumber\\
=&\frac{-4\epsilon+O(x_1^{2+\gamma})}{\sigma+x_1^2}+O(\epsilon)\ \mbox{on}\ \tilde\Gamma_d^+
\end{align*}
and
\begin{equation*}
\left(\frac{1}{D_1}+\frac{1}{D_2}\right)=\frac{2+O(\sqrt{\epsilon})}{\sigma+x_1^2} ,  
\end{equation*}
we have
\begin{equation*}
\frac{\partial\tilde\varphi}{\partial\nu}\Big|_{\tilde{\Gamma}_d^+}
\leq\frac{1}{\sqrt{1+(h_1')^2}}\left\{
\left(-2\kappa_1x_1^2-\epsilon+O(x_1^{2+\gamma})\right)\frac{2}{\sigma+x_1^2}+O(\sqrt{\epsilon})
\right\}\leq0.   
\end{equation*}
The value of $\frac{\partial\tilde\varphi}{\partial\nu}$ in $\tilde\Gamma_d^-$  can be proved similar.

Then, similar to the proof of Proposition \ref{propuhat}, taking $d=\sqrt{\epsilon}$, we get
\begin{equation*}
    |u(x)|\leq \frac{C\sqrt{\sigma}|x_1|}{x_1^2+\sigma}\quad\mbox{in}\ \tilde\Omega_{\sqrt{\epsilon}}.
\end{equation*}
and
\begin{align*}
u(x)\geq&\ln(x_1^2+(x_2+\sqrt{\epsilon})^2)+\ln(x_1^2+(x_2-\sqrt{\epsilon})^2)-2\ln\epsilon\nonumber\\
\geq& \frac{1}{C}\ln\left(\frac{\epsilon+x_1^2+x_2^2}{\epsilon}\right)\quad\mbox{in}\ \tilde\Omega_{\sqrt{\epsilon}}.
\end{align*}
\textbf{Case 2: for general $\kappa_1,\kappa_2$.}
 We denote 
 $$
y_1=x_1,\quad y_2=x_2-\frac{h_1(y_1)+h_2(y_1)}{2},
 $$
consider the supersolution
$$
\varphi_1(y)=\varphi(y_1,y_2)\quad \mbox{in}\ \tilde\Omega_{\sqrt{\epsilon}}
$$
and the subsolution
$$
\varphi_2(y)=\tilde\varphi(y_1.y_2)\quad \mbox{in}\ \tilde\Omega_{\sqrt{\epsilon}}
$$
with 
$$\sigma=\frac{\epsilon}{\kappa}.
$$
The proof is finished.
\end{proof}

\begin{proof}[Proof of Theorem \ref{propu}]
By the maximum principle and Hopf Lemma , we know that $u$ is bounded in $\Omega_R$. 
For $n=2$, in view of Lemma \ref{u_n=2}, we know 
$|u(\sqrt{\epsilon})|=O(1)$. Hence, we know $u$ is the polynomial of order $0$ in $\Omega_R$ for $n=2$.

For $n\geq3$,  we denote $Y_{k,i}$ to be a $k$-th degree normalized spherical harmonics so that $\{Y_{k,i}\}_{k,i}$ forms an orthonormal basis of $L^2(\mathbb{S}^{n-2})$, and then we decompose $u$ by
$$
u(x)=u(r,\xi,x_n)=u(0,x_n)+\sum_{k=1}^{\infty}\sum_{i=1}^{N(k)}U_{k,i}(r,x_n)Y_{k,i}(\xi),
$$
where we used the assumption that $u(0)=0$. Hence,
$$
U_{k,i}(r,x_n)
=\int_{\mathbb{S}^{n-2}}u(r,\xi,x_n)Y_{k,i}(\xi)d\xi.
$$
Hence, in view of Proposition \ref{propuhat}, one has
$$
|U_{k,i}(r,x_n)|= |\hat u(r,x_n)|\leq C\left(r^{\alpha_k}+\epsilon\right),
$$
where $C$ depends only on $n,\gamma,\kappa,R$.
Thus, for the fixed point $x_0\in\Omega_{R/2}$,
\begin{equation*}
\left(\fint_{\partial \Omega_{\rho}(x_0)}| u(x)-u(0,x_n)|^2\right)^{1/2}=\left(\sum_{k=1}^{\infty}\sum_{i=1}^{N(k)}|U_{k,i}(\rho,x_n)|^2\right)^{1/2}\leq C\left(\rho^{\alpha}+\epsilon\right).    
\end{equation*}
for $0<\rho<R/2$. By Moser iteration (see e.g. Theorem 8.17 in \cite{gilbarg1977elliptic}), we know
\begin{equation}\label{ubar}
\| u-u(0,x_n)\|_{L^\infty( \Omega_\rho(x_0))}\leq C\left(\fint_{\partial \Omega_{\rho}(x_0)}|u(x)-u(0,x_n)|^2\right)^{1/2}
\leq C\left(\rho^{\alpha}+\epsilon\right),
\end{equation}
for $0<\rho<R/2$, where $C$ depends only on $n,\gamma,\kappa,R$. Taking $\rho=\sqrt{\epsilon+|x_0'|^2+x_{0,n}^2}$ in \eqref{ubar}, we can deduce that 
$$
|u(x)-u(0,x_n)|\leq C(\epsilon+|x'|^2+x_n^2)^{\alpha/2}\quad\mbox{in}\ \Omega_{R/2}.
$$
In view of Corollary \ref{Cor:u_n}, we have
\begin{equation*}
|u(0,x_n)|=|u(0,x_n)-u(0)|\leq|\partial_nu||x_n|\leq C\epsilon.    
\end{equation*}
Hence, from the above, we have
\begin{equation*}
 |u(x)|\leq C(\epsilon+|x'|^2+x_n^2)^{\alpha/2}\quad\mbox{in}\ \Omega_{R/2}.   
\end{equation*}
Similarly,  using the lower bound of \eqref{uhat}, we have
\begin{equation*}
|u(x)|\geq\frac{1}{C}(|x'|^2+x_n^2)^{\alpha/2}\quad\mbox{for any}\ x\in \Omega_{R/2}.
\end{equation*}
The proof is completed.
\end{proof}

\section{Proof of Theorem \ref{grad_pointwise estimates} and Theorem \ref{grad_pointwise estimates2}.}\label{sec3}
\begin{proof}[Proof of Theorem \ref{grad_pointwise estimates}]

In view of Lemma \ref{grad_local estimates} and Theorem \ref{propu}, one know
\begin{equation*}
|\nabla u(x)|
\leq C(\epsilon+|x'|^2)^{\frac{\alpha-1}{2}}\quad\mbox{for any}\ x\in\Omega_{R/2},    
\end{equation*}
where $C$ depends only on $n,\gamma,\kappa,\rho,\|g\|_{L^\infty(\partial D)}$.

On the other hand, using the differential median theorem, taking suitable $\rho=O(\sqrt{\epsilon+|x_0'|^2})$, one has
$$
|\nabla u(x_0)|\geq|\nabla_{x'} u(x_0)|\geq\frac{|u(x)-u(y)|}{|x'-y'|}\geq\frac{1}{C}(\epsilon+|x_0'|^2)^{(\alpha-1)/2}\ 
\mbox{for}\ x'\neq y',
$$
where $x, y\in \Omega_{\rho/2}(x_0)$, $C$ depends only on $n,\gamma,\kappa,R,\|g\|_{L^\infty(\partial D)}$.
The proof is completed.
\end{proof}	

\begin{proof}[Proof of Theorem \ref{grad_pointwise estimates2}]
After a suitable rotation, without loss of generality, we may assume that $\mu_1\geq\mu_2\geq\ldots\geq\mu_{n-1}>0$ in assumption \eqref{assume6}.
Then by a coordinate rotation, the boundary assumptions can be converted into
\begin{equation}\label{h}
(h_1-h_2)(x')=a(\xi)|x'|^2+O(|x'|^{2+\gamma}),\quad|x'|<R  
\end{equation}
with $\xi\in\mathbb{S}^{n-2}$ and $a(\xi)>0$. Meanwhile, the equation can be written as follows:
$$
\operatorname{div}(a(\xi)\nabla u(r,\xi,x_n))=0\quad\mbox{in}\ \Omega_{R},
$$
that is 
\begin{equation*}
u_{rr}+\frac{n-2}{r}u_r-\frac{1}{a(\xi)r^2}\operatorname{div}_{\mathbb{S}^{n-2}}(a(\xi)\nabla_{\mathbb{S}^{n-2}}u)+u_{nn}=0\quad \mbox{in}\ \Omega_{R}.
\end{equation*}
We still denote $\hat{v}:=\hat{u}_{1,i}$ as in \eqref{def_uhat}, then $\hat{v}$ is the solution of the following
\begin{equation*}
\begin{cases}
\hat{v}_{rr}+\frac{n-2}{r}\hat{v}_r-\frac{\lambda}{r^2}\hat{v}+\hat{v}_{nn}=0 &\mbox{in}\ \tilde{\Omega}_{R},\\
\frac{\partial\hat{v}}{\partial\nu}=0&\mbox{on}\ \tilde{\Gamma}_R^\pm,\\
\hat{v}(0)=0,
\end{cases}
\end{equation*} 
where $\lambda$ is the first nonzero eigenvalue of the equation 
$$
-\operatorname{div}_{\mathbb{S}^{n-2}}(a(\xi)\nabla_{\mathbb{S}^{n-2}}u)=\lambda a(\xi)u,\quad\xi\in\mathbb{S}^{n-2}.
$$
Then using \eqref{h}, similar to the proof of Theorem \ref{grad_pointwise estimates}, we get the results.
\end{proof}

\section{Proof of Lemma \ref{grad_local estimates}}\label{sec4}

In this section, we provide the proof of Lemma \ref{grad_local estimates}. The application of  Lemma \ref{grad_nabla_u} is essential.
It is possible to regard this as constituting a generation of Theorem 4.1 in \cite{weinkove2023insulated}, wherein the author exclusively addresses the case of two ball inclusions.
\begin{lemma}\label{grad_nabla_u}
Assume that $u$ satisfies $\frac{\partial u}{\partial\nu}=0$ on $\Gamma_R^\pm$.
Under the assumption of \eqref{assume1}-\eqref{assume3},\eqref{h_1h_2}, we have
\begin{equation}\label{grad_nabla_u_eq}
\frac{\partial}{\partial\nu}(|\nabla u|^{2})
=\frac{1}{\sqrt{1+\sum_{i=1}^{n-1}(\partial_ih_j)^2}}\Big(4\kappa_j +O(|x'|^\gamma)\Big) |\nabla u|^{2}\ \ \mbox{on}\ \ \Gamma_R^\pm,
\end{equation}
where $j=1$ when $x\in\Gamma_R^+$, $j=2$ when $x\in\Gamma_R^-$.
\end{lemma}
\begin{proof}
Observe that the outward pointing unit normal of $\Gamma_R^+$ is given by
\begin{equation}\label{vector}
\nu=\frac{1}{\sqrt{1+\sum_{i=1}^{n-1}(\partial_ih_1)^2}}\left(-\partial_1h_1,\cdots,-\partial_{n-1}h_1,1 \right).
\end{equation}
So the boundary data $\frac{\partial u}{\partial\nu}=0$ on $\Gamma_R^+$ is equivalent to
\begin{equation}\label{boundary1}
	\sum_{i=1}^{n-1}\partial_ih_1\partial_iu=\partial_nu\quad\mbox{on}\ \Gamma_R^+.
\end{equation}
Differentiating the above with respect to $x_j$ gives
\begin{equation*}
\sum_{i=1}^{n-1}\partial_{ij}h_1\partial_iu+\sum_{i=1}^{n-1}\partial_ih_1\partial_{ij}u+\sum_{i=1}^{n-1}\partial_{in}u\partial_ih_1\partial_jh_1=\partial_{nj}u+\partial_{nn}u\partial_jh_1 \ \mbox{on}\ \Gamma_R^+ 
\end{equation*}
for $j=1,\ldots,n-1$.
From the above, we have that 
\begin{equation}\label{boundaryx_i}
-\sum_{i=1}^{n-1}\partial_ih_1\partial_{ij}u=\sum_{i=1}^{n-1}\partial_{ij}h_1\partial_iu+\sum_{i=1}^{n-1}\partial_{in}u\partial_ih_1\partial_jh_1-\partial_{nj}u-\partial_{nn}u\partial_jh_1     \ \mbox{on}\ \Gamma_R^+.
\end{equation}
Then using \eqref{boundary1} and \eqref{boundaryx_i}, one has 
		\begin{align}\label{eq2.2}
		\frac{\partial}{\partial\nu}(|\nabla u|^2)
        =&\frac{\partial}{\partial\nu}\left(\sum_{j=1}^{n}(\partial_j u)^2\right)\nonumber\\
        =&\frac{1}{\sqrt{1+\sum_{k=1}^{n-1}(\partial_ih_1)^2}}\left(-2\sum_{i=1}^{n-1}\sum_{j=1}^n\partial_ju\partial_ih_1\partial_{ij}u+2\sum_{j=1}^n\partial_ju\partial_{jn}u\right)\nonumber\\
	=&\frac{1}{\sqrt{1+\sum_{k=1}^{n-1}(\partial_ih_1)^2}}\Bigg[2\sum_{j=1}^{n-1}\partial_ju\Big(\sum_{i=1}^{n-1}\partial_{ij}h_1\partial_iu+\sum_{i=1}^{n-1}\partial_{in}u\partial_ih_1\partial_jh_1\nonumber\\
    &-\partial_{nj}u-\partial_{nn}u\partial_jh_1\Big)
    -2\sum_{i=1}^{n-1}\partial_nu\partial_{ni}u\partial_ih_1+2\sum_{j=1}^n\partial_ju\partial_{jn}u\Bigg]\nonumber\\
	=&\frac{2}{\sqrt{1+\sum_{k=1}^{n-1}(\partial_ih_1)^2}}\sum_{i,j=1}^{n-1}\partial_{ij} h_1\partial_iu\partial_ju\nonumber\\
        =&\frac{1}{\sqrt{1+\sum_{i=1}^{n-1}(\partial_ih_1)^2}}\Big(4\kappa_1 +O(|x'|^\gamma)\Big) |\nabla u|^{2}\ \ \mbox{on}\ \ \Gamma_R^+.
	   \end{align}
Hence, \eqref{grad_nabla_u_eq} holds for $x\in\Gamma_R^+$. Similarly, we can prove it for $x\in\Gamma_R^-$. The proof is finished.
\end{proof}
For the flat boundary, the conclusion is much different, which is as follows.
\begin{corollary}
 \label{grad_nabla_u_flat}
Assume that $u$ satisfies $\frac{\partial u}{\partial\nu}=0$ on $\Gamma_R^\pm$.
Under the assumption of \eqref{assume4}, we have
\begin{equation*}
\frac{\partial}{\partial\nu}(|\nabla u|^{2})
=0\ \ \mbox{on}\ \ \Gamma_R^\pm.
\end{equation*}   
\end{corollary}

\begin{proof}[Proof of Lemma \ref{grad_local estimates}.]
We use the method of \cite{weinkove2023insulated} to give the proof of \eqref{nabla_u} for any general inclusions.  
In the following, we give the proof into two cases.

\textbf{Case 1: $\kappa_1>0$ and $\kappa_2<0$.}
We define
\begin{equation}\label{Q}
Q:=(\epsilon+|x'|^2-Ax_n^2)|\nabla u|^2+Bu^2
\end{equation}
with
\begin{equation}\label{AB}
A>8\max\{\kappa_1,-\kappa_2\},\quad B>A-n+2.  
\end{equation}
In the following, we prove that $Q$ attains its maximum at some point $P$ on $\partial\Omega_\rho\setminus\Gamma_\rho^\pm$.

Let $Q$ achieve its maximum on $\Omega_\rho$ at a point $p$. Firstly, if $p$ is on $\Gamma_\rho^+$. Then using Lemma \ref{grad_nabla_u}, 
\begin{align*}
0\leq\frac{\partial Q}{\partial\nu}\Big|_{\Gamma_\rho^+}
=&(\epsilon+|x'|^2-Ax_n^2)\frac{\partial}{\partial\nu}(|\nabla u|^2)+\left(\frac{\partial|x'|^2}{\partial\nu}-A\frac{\partial x_n^2}{\partial\nu}\right)|\nabla u|^2\nonumber\\
=&\frac{1}{\sqrt{1+\sum_{i=1}^{n-1}(\partial_ih_1)^2}}\Big\{[4\kappa_1-A+O(|x'|^\gamma)]\epsilon-[2A\kappa_1+O(|x'|^\gamma)]|x'|^2\nonumber\\
&-4A\kappa_1x_n^2[1+O(|x'|^\gamma)]\Big\}|\nabla u|^2.
\end{align*}
Shrinking $\rho$ if necessary, under the assumption \eqref{AB}, we have
$$
4\kappa_1-A+O(|x'|^\gamma)<0,\quad
-[2A\kappa_1+O(|x'|^\gamma)]<0,\quad-4A\kappa_1x_n^2[1+O(|x'|^\gamma)]<0.
$$
Hence, we get
$$
\frac{\partial Q}{\partial\nu}\Big|_{\Gamma_\rho^+}(p)<0, 
$$
which is a contradiction. 

Next, if $p$ is on $\Gamma_\rho^-$, similar to the above, one has 
\begin{align*}
0\leq\frac{\partial Q}{\partial\nu}\Big|_{\Gamma_\rho^-}
=&\frac{1}{\sqrt{1+\sum_{i=1}^{n-1}(\partial_ih_2)^2}}\Big\{[4\kappa_2-A+O(|x'|^\gamma)]\epsilon+(8+2A)\kappa_2|x'|^2\nonumber\\
&+O(|x'|^{2+\gamma})
-4A\kappa_2x_n^2(1+O(|x'|^\gamma))\Big\}|\nabla u|^2,
\end{align*}
since the term $O(|x'|^{2+\gamma})
-4A\kappa_2x_n^2(1+O(|x'|^\gamma))$ is small compared to the first two term, shrinking $\rho$ if necessary, we have
$$
\frac{\partial Q}{\partial\nu}\Big|_{\Gamma_\rho^-}(p)<0, 
$$
which is a contradiction.

Next, we assume that $p$ is an interior point of $\Omega_\rho$. Then we have
\begin{align*}
0\geq\Delta Q=&\Delta(\epsilon+|x'|^2-Ax_n^2)|\nabla u|^2+(\epsilon+|x'|^2-Ax_n^2)\Delta(|\nabla u|^2)\nonumber\\
&+2\nabla(\epsilon+|x'|^2-Ax_n^2)\cdot\nabla(|\nabla u|^2)+2B|\nabla u|^2\nonumber\\
=&2(n-1-A+B)|\nabla u|^2+2(\epsilon+|x'|^2-Ax_n^2)|\nabla\nabla u|^2\nonumber\\
&+2\nabla(\epsilon+|x'|^2-Ax_n^2)\cdot\nabla(|\nabla u|^2).
\end{align*}
Since
\begin{equation*}
2\nabla|x'|^2\cdot\nabla(|\nabla u|^2)
\geq
-8|x'||\nabla u||\nabla\nabla u|
\geq-|x'|^2|\nabla\nabla u|^2-16|\nabla u|^2,
\end{equation*}
\begin{equation*}
    -2A\nabla x_n^2\cdot\nabla(|\nabla u|^2)\geq-4A|x_n||\nabla u||\nabla\nabla u|\geq -4A^2x_n^2|\nabla\nabla u|^2-|\nabla u|^2,
\end{equation*}
in view of \eqref{AB}, shrinking $\rho$ if necessary, one has
\begin{equation*}
0\geq\Delta Q\geq2(n-2-A+B)|\nabla u|^2+[2\epsilon+|x'|^2-2A(1+2A)x_n^2]|\nabla\nabla u|^2>0\quad\mbox{in}\ \Omega_\rho ,   
\end{equation*}
where we used the fact that $x_n^2$ is small compared to $\epsilon+|x'|^2$.
This is a contradiction.

It follows that $Q$ achieve its maximum at some point $p$ on $\partial\Omega_\rho\setminus\Gamma_\rho^\pm$, that is
\begin{equation}\label{Q_upper bound}
    Q(x)\leq Q(x)\Big|_{\partial\Omega_\rho\setminus\Gamma_\rho^{\pm}}\quad\mbox{for any}\ x\in\Omega_{\rho}.
\end{equation}
In view of the maximum principle, we have $\|u\|_{L^\infty(D_0)}\leq \|g\|_{L^\infty(\partial D)}$. Taking $\rho=R/2$, from \eqref{Q_upper bound} and the classical elliptic estimate, we have
\begin{equation*}
Q=(\epsilon+|x'|^2-2x_n^2)|\nabla u|^2+Au^2\leq C\|u\|_{L^\infty(\Omega_{R})}\leq C\|g\|_{L^\infty(\partial D)}\quad\mbox{for any}\ \Omega_{R/2},
\end{equation*}
where $C$ depends only on $n,\kappa,R$.
That is, 
\begin{equation*}
|\nabla u(x)|\leq\frac{C}{\sqrt{\epsilon+|x'|^2}}\quad\mbox{for any}\ x\in \Omega_{R/2},   
\end{equation*}
where $C$ depends only on $n, \kappa,R$ and $\|g\|_{L^\infty(\partial D)}$. 

Next,  we give the proof of \eqref{nabla_u}. In view of the boundary condition, we know
$$
|\partial_nu(x)|
\leq C|x'||\partial_{x'}u(x)|
\leq C
\quad\mbox{for any}\ x\in \Gamma_R^\pm,
$$
where $C$ depends on $n,\kappa,R$ and $\|g\|_{L^\infty(\partial D)}$. 
Then by the maximum principle and \eqref{D0}, we know
\begin{equation}\label{gradu_n}
|\partial_nu(x)|\leq C\quad\mbox{for any}\ x\in \Omega_{R},     
\end{equation}
where $C$ depends on $n,\kappa,\gamma,R$ and $\|g\|_{L^\infty(\partial D)}$. Denote 
$$
w(x'):=\sup_{x_n\in(-\frac{\epsilon}{2}+h_2(x'),\frac{\epsilon}{2}+h_1(x'))}|u(x',x_n)|,
$$
one has, for $|x'|<R$,
$$
|w(x')|\leq|u(x)|+|w(x')-u(x)|\leq C(|u(x)|+|\partial_nu||x_n|)\leq C(|u(x)|+|x_n|).
$$
Hence, in view of the classical elliptic estimate (see e.g. Theorem 6.26 in \cite{gilbarg1977elliptic}),  we have
$$
|\nabla_{x'} u(x)|\leq |\nabla w(x')|\leq \frac{C}{\rho}\|w\|_{L^\infty(B_{2\rho}(x))}
\leq\frac{C}{\rho}\left(\|u\|_{L^\infty(\Omega_{2\rho}(x))}+\epsilon+\rho^2\right)\ \mbox{in}\ \Omega_{2\rho}(x).
$$
Hence, in view of \eqref{gradu_n},
$$
|\nabla u(x)|\leq |\partial_{x'} u(x)|+|\partial_n u(x)|
\leq \frac{C}{\rho}\left(\|u\|_{L^\infty(\Omega_{2\rho}(x))}+\epsilon+\rho^2\right)\ \mbox{in}\ \Omega_{2\rho}(x).
$$
Taking $\rho=\sqrt{\epsilon+|x'|^2}$, we have \eqref{nabla_u}. 

\textbf{Case 2: $\kappa_1>0$ and $\kappa_2\geq0$.} We can modify the definition of $Q$ as follows:
\begin{equation}\label{Q2}
Q:=\left[\epsilon+|x'|^2-A_1\left(x_n-\frac{\epsilon}{2}-h_1(x')\right)^2-A_2\left(x_n+\frac{\epsilon}{2}-h_2(x')\right)^2\right]|\nabla u|^2+Bu^2
\end{equation}
in $\Omega_\rho$ with
\begin{equation}\label{A_1A_2B}
A_1>8\kappa_2/\kappa,\quad A_2>4\kappa_1,\quad B>A_1+A_2-n+2.  
\end{equation}

Let $Q$ achieve its maximum on $\Omega_\rho$ at a point $p$. Firstly, if $p$ is on $\Gamma_\rho^+$,  using Lemma \ref{grad_nabla_u}, in view of $x_n=\frac{\epsilon}{2}+h_1(x')$ on $\Gamma_\rho^+$, we have
\begin{align*}
0\leq\frac{\partial Q}{\partial\nu}\Big|_{\Gamma_\rho^+}
=&\frac{1}{\sqrt{1+\sum_{i=1}^{n-1}(\partial_ih_1)^2}}\Big\{[4\kappa_1-2A_2+O(|x'|^\gamma)]\epsilon\nonumber\\&-[2A_2\kappa+O(|x'|^\gamma)]|x'|^2\Big\}|\nabla u|^2.
\end{align*}
Shrinking $\rho$ if necessary, under the assumption \eqref{A_1A_2B}, we have
$$
4\kappa_1-2A_2+O(|x'|^\gamma)<0,\quad
-[2A_2\kappa+O(|x'|^\gamma)]<0.
$$
Hence, we get
$$
\frac{\partial Q}{\partial\nu}\Big|_{\Gamma_\rho^+}(p)<0, 
$$
which is a contradiction. 

Next, if $p$ is on $\Gamma_\rho^-$, similar to the above, shrinking $\rho$ if necessary, one has 
\begin{align*}
0\leq\frac{\partial Q}{\partial\nu}\Big|_{\Gamma_\rho^-}
=&\frac{1}{\sqrt{1+\sum_{i=1}^{n-1}(\partial_ih_2)^2}}\Big\{[4\kappa_2-2A_1+O(|x'|^\gamma)]\epsilon\nonumber\\
&+(8\kappa_2-2A_1\kappa+O(|x'|^\gamma))|x'|^2\Big\}|\nabla u|^2<0,
\end{align*}
which is a contradiction. 

Next, we assume that $p$ is an interior point of $\Omega_\rho$. Then 
in view of \eqref{A_1A_2B}, shrinking $\rho$ if necessary, one has
\begin{equation*}
0\geq\Delta Q\geq2(n-2-A_1-A_2+B)|\nabla u|^2+[2\epsilon+|x'|^2+O((\epsilon+|x'|^2)^2)]|\nabla\nabla u|^2>0\ \mbox{in}\ \Omega_\rho ,   
\end{equation*}
where we used the fact that $O((\epsilon+|x'|^2)^2)$ is small compared to $\epsilon+|x'|^2$.
This is a contradiction.

It follows that $Q$ achieve its maximum at some point $p$ on $\partial\Omega_\rho\setminus\Gamma_\rho^\pm$. Then similar to Case 1, we get the result.

The proof is finished.
\end{proof}

\section{Proof of Theorem \ref{flat}}\label{sec5}
\begin{lemma}\label{grad_local estimates_flat}
Suppose $u\in H^1(\Omega_R)$  is a solution of system  \eqref{ins1} with $h_1(x')$ and $h_2(x')$ satisfying \eqref{assume4}, 
then there exists a positive constant $C$ which depends only on $n,\gamma,\kappa,R,\|g\|_{L^\infty(\partial D)}$, such that for any $\rho\in(0,R/2)$,
 \begin{equation*}
|\nabla u(x)|\leq \frac{C|u(x)|}{\sqrt{\epsilon+|x'|^2}}\quad \mbox{for any}\quad x\in\Omega_{R/2}.
\end{equation*} 
\end{lemma}
\begin{proof}
By using Corollary \ref{grad_nabla_u_flat}, the proof is similar to the proof of Lemma \ref{grad_local estimates}. We omit it here.   
\end{proof}

\begin{proof}[Proof of Theorem \ref{flat}]
Let $\hat{v}:=\hat{u}_{1,i}$ be defined by \eqref{def_uhat}, then we know it satisfies the following equation:
\begin{equation}\label{eq:uhat_flat}
\begin{cases}
\hat{v}_{rr}+\frac{n-2}{r}\hat{v}_r-\frac{n-2}{r^2}\hat{v}+\hat{v}_{nn}=0 &\mbox{in}\ \tilde{\Omega}_{R},\\
\partial_n\hat{v}=0&\mbox{on}\ \tilde{\Gamma}_R^\pm,\\
\hat{v}(0)=0,
\end{cases}
\end{equation}
with 
\begin{equation*}
\tilde{\Omega}_R:=\left\lbrace (r,x_n)\in \mathbb{R}^2\ |\ 0< r<R,\  -\frac{\epsilon}{2}<x_n<\frac{\epsilon}{2}\right\rbrace,
\end{equation*}
\begin{equation*}
\tilde{\Gamma}_R^+:=\left\lbrace (r,x_n)\in \mathbb{R}^2\ |\ 0< r<R,\ x_n=\frac{\varepsilon}{2}\right\rbrace, 
\end{equation*}
\begin{equation*}
\tilde{\Gamma}_R^-:=\left\lbrace (r,x_n)\in \mathbb{R}^2\ |\ 0< r<R,\  x_n=-\frac{\varepsilon}{2}\right\rbrace.
\end{equation*}
We find that 
$$
\hat{v}(r,x_n)=r
$$
is a solution of \eqref{eq:uhat_flat}.
Then similar to the proof of Theorem \ref{grad_pointwise estimates}, we get
$$
\frac{1}{C}|x'|\leq |u(x)|\leq C\sqrt{\epsilon+|x'|^2}\quad\mbox{for any}\ x\in\Omega_{R/2}.
$$
Combining the above with Lemma \ref{grad_local estimates_flat}, we get 
\begin{equation*}
|\nabla u(x)|
\leq C\quad\mbox{for any}\ x\in\Omega_{R/2}, 
\end{equation*}
where $C$ depends only on $n,\gamma,\kappa,R,\|g\|_{L^{\infty}(\partial D)}$.
The proof is finished. 
\end{proof}

\thanks{\textbf{Acknowledgment}}
The author thank Professor Liming Sun and Zhitao Zhang for his comments and suggestions on this article. The author would like to thank the anonymous reviewers and mathematics enthusiasts for their suggestions on the content and writing of the article.

\small
\bibliographystyle{plainnat}
\bibliography{conductivity_ref}

\end{document}